\documentclass[a4paper,12pt]{amsart}

\usepackage{fullpage}
\usepackage{latexsym,amsfonts,amssymb,amsmath,amsthm, color}
\usepackage{microtype}

\usepackage[top=1in, bottom=1in, left=1in, right=1in]{geometry}
\usepackage{bbm}
\usepackage{mathrsfs,bm}

\usepackage{url}
\usepackage[english]{babel}
\usepackage{paralist}
\usepackage{amsthm,amsfonts,amssymb,amsmath}
\usepackage{comment}

\usepackage[latin1]{inputenc}

\newcommand{\tRe}{\textup{Re }}

\newcommand{\R}{\mathbb{R}}

\newcommand{\C}{\mathbb{C}}

\newcommand{\es}[1]{\begin{equation}\begin{split}#1\end{split}\end{equation}}
\newcommand{\est}[1]{\begin{equation*}\begin{split}#1\end{split}\end{equation*}}

\renewcommand{\mod}[1]{~\pr{\textnormal{mod}~#1}}
\newtheorem{thm}{Theorem}[section]
\newtheorem{corollary}[thm]{Corollary}

\newtheorem{lem}[thm]{Lemma}

\newtheorem{defi}[thm]{Definition}
\theoremstyle{remark}
\newtheorem{rem}{Remark}

\newtheorem{rem*}{Remark}
\newcommand{\pr}[1]{\left( #1\right)}

\let\originalleft\left
\let\originalright\right
\renewcommand{\left}{\mathopen{}\mathclose\bgroup\originalleft}
\renewcommand{\right}{\aftergroup\egroup\originalright}

\numberwithin{equation}{section}

\newcommand{\Rep}{\textrm{Re}}

\begin{document}
\title{On Benford's Law for multiplicative functions}

\date{
\today}

\author[V. Chandee]{Vorrapan Chandee}
\address{Mathematics Department \\ Kansas State University \\ Manhattan, KS 66503}
\email{chandee@ksu.edu}

\author[X. Li]{Xiannan Li}
\address{Mathematics Department \\ Kansas State University \\ Manhattan, KS 66503}
\email{xiannan@math.ksu.edu }

\author{Paul Pollack}
\address{Mathematics Department \\ University of Georgia \\ Athens, GA 30602}
\email{pollack@uga.edu}

\author{Akash Singha Roy}
\address{ESIC Staff Quarters No.: D2\\ 143 Sterling Road, Nungambakkam\\Chennai 600034\\ Tamil Nadu, India}
\email{akash01s.roy@gmail.com}


\allowdisplaybreaks
\numberwithin{equation}{section}
\begin{abstract} We provide a criterion to determine whether a real multiplicative function is a strong Benford sequence.  The criterion implies that the $k$-divisor functions, where $k \neq 10^j$, and Hecke eigenvalues of newforms, such as Ramanujan tau function, are strong Benford. Moreover, we deduce from the criterion that the collection of multiplicative functions which are not strong Benford forms a group under pointwise multiplication.  In contrast to earlier work, our approach is based on Hal\'{a}sz's Theorem. 

\end{abstract}

\maketitle

\section{Introduction}

Benford's law is a phenomenon about the first digits of the numbers in data sets. In particular, the leading digits do not exhibit uniform distribution as might be naively expected, but rather, the digit $1$ appears the most, followed by $2, 3$, and so on until $9$. More precisely, we define Benford's law for a sequence of numbers below.
 
\begin{defi} Let $\{a_k\}$ be a sequence of { positive} real numbers. Suppose that 
$$a_k = 10^{M_k }\sum_{j = 0}^\infty d_j^{(k)}10^{-j}, $$
where $M_k \in \mathbb{Z}$, $1 \leq d_0^{(k)} \leq 9,$ and $0 \leq d_j^{(k)} \leq 9$ for all $j \geq 1.$ $\{a_k\}$ satisfies strong Benford's law and is called a {\it strong Benford sequence} or {\it strong Benford}, if for all strings $S = d_0d_1...d_{K-1}$, where $d_0 \neq 0$, and for all positive integers $K$, 
\est{ \lim_{N \rightarrow \infty} \frac{\# \left\{ k \leq N \ | \  d_{0}^{(k)} d_1^{(k)} d_2^{(k)} .... d_{K-1}^{(k)} = d_0d_1d_2...d_{K-1}\right\}}{N} &= \log_{10}\left( r + 10^{-K + 1}\right) - \log_{10} r \\
&= \log_{10}\left( 1 + \frac 1S\right),}
where $r = \sum_{j= 0}^{K-1} d_j 10^{-j}.$

\end{defi}

\begin{defi} Let $\{b_k\}$ be a sequence of \underline{non-negative} real numbers. Let $\{b_j'\}$ be the subsequence of $\{b_k\}$ obtained by removing all zero terms. We say that $\{b_k\}$ is strong Benford if $\{b_j'\}$ is strong Benford. 
\end{defi}

\begin{rem}

When $K = 1,$ the condition above becomes that for $d = 1, 2, ..., 9,$ 
$$ \lim_{N \rightarrow \infty} \frac{ \# \left\{  k \leq N  \ | \ d_0^{(k)} = d\right\}}{N} = \log_{10}(d + 1) - \log_{10} d = \log_{10} \left( 1 + \frac 1d\right).$$
From now on, we will  refer to {\it strong Benford's law} simply as ``Benford's law.''

\end{rem}
If a sequence of numbers is Benford, then the probability that its leading digit is 1 is   $\log 2 \approx 30.1 \%$ while that of $9$ as a leading digit is only $\log(1+1/9) \approx 4.6\%$.  The Benford's law phenomena was first observed by the astronomer Simon Newcomb in 1881 \cite{Newcomb} when he noticed that the first pages of a book of logarithm tables were the most worn. Later, in 1938, Frank Benford \cite{Benford} discovered similar patterns and found numerical evidence from a multitude of data sets, e.g. population numbers, areas of rivers, and physical constants. 

A number of familiar sequences in mathematics such as Fibonacci sequences, exponential functions, and factorial functions have been proven to follow Benford's law.  Sequences in number theory appear as well.  For instance, Kontorovich and Miller \cite{KM} showed that the distribution of values of L-functions and certain statistics concerning the iterates of the $3x+1$ problem follow Benford's law.  A refinement of Lagarias and Soundararajan \cite{LS} proved that iterates of $3x+1$ problem follow Benford's law for most initial seeds. Recently, the first author and Aursukaree \cite{AC} proved that the divisor function, which counts the number of positive divisors of $n$, is Benford. 
The proof of this was an application of the Selberg-Delange method.

The divisor function is an example of the wider class of real multiplicative functions.  We say that $h$ is a real multiplicative function if $h\colon \mathbb N \rightarrow \mathbb R$ satisfies that for all natural numbers $m$ and $n$ with $(m, n) = 1,$ 
$$h(mn) = h(m) h(n). $$ 
In this article, we will give a criterion for when real multiplicative functions follow Benford's law. Then we will apply the criterion to a number of interesting examples, including the $k$-divisor functions, Euler's Phi function and Hecke eigenvalues of new forms. 

The proof of our criterion is a nice application of Hal\'asz's theorem on the sum of multiplicative functions. Roughly speaking, we find that $\{h(n)\}$ is a strong Benford sequence if and only if $e^{2\pi i \ell \log_{10} |h(n)|}$ is not ``close to'' $n^{i\alpha}$ for all nonzero integers $\ell$ and all real $\alpha.$ Another important feature of the criterion is that we can determine if $\{h(n)\}$ is Benford as soon as information at prime numbers is known.  In Section \ref{sec:MainThm}, we will provide the statement of Hal\'{a}sz's theorem and Weyl's criterion and then explicitly state and prove our main criterion (Theorem \ref{thm:Benfordcriterion}).  Before that, we illustrate our criterion with some applications below.

\subsection{Applications of the criterion - Theorem \ref{thm:Benfordcriterion}} 
We start with the sequence $\{ n^a\}$ where $a$ is fixed. Previously, this sequence has been proven to be not strong Benford by other methods, e.g. Fejer's Theorem (see \cite{Zheng}). Applying Theorem \ref{thm:Benfordcriterion}, we have an alternative simple proof of the same result.
\begin{corollary} \label{cor:poly} Let $a$ be a fixed real constant. Then  
$\{n^a\}$ is not a strong Benford sequence.
\end{corollary}

Next, we will consider the $k$-divisor function, which counts the number of ways to write $n$ as the product of $k$ natural numbers, i.e.
$$ d_k(n) = \sum_{n_1n_2....n_k = n}  \ 1.$$  
As previously mentioned, $\{d_2(n)\} = \{d(n)\}$ is a strong Benford sequence. The first author and Aursukaree used Selberg-Delange's method \cite[Chapter 7]{MV}, which involves analysis with contour integrals. Theorem \ref{thm:Benfordcriterion} applies to $d_k(n)$ for general $k$ to give an alternate simpler proof. 
\begin{corollary} \label{cor:dkn}
Let $k \geq 2$, and $d_k(n)$ be the k-divisor function.
Then $\{d_k(n)\}$ is a strong Benford sequence if and only if $k \neq 10^{j}$ for all positive integers $j$. 
\end{corollary}
Note that when $k$ is a power of 10,
the first digit of $d_k(n)$ is 1 for all squarefree integers $n$. Since the number of squarefree integers up to $x$ is about $\frac{6}{\pi^2} x$, the probability that the first digit of $d_k(n)$ being $1$ exceeds $60.8\% > 30.1\%$.

Another important multiplicative function is Euler's phi function $\varphi(n)$, which counts the number of positive integers up to $n$ that are relatively prime to $n$. The formula is
$$ \varphi(n) = n \prod_{p | n} \left( 1 - \frac 1p\right).$$
\begin{corollary} \label{cor:phi}
Let $\varphi(n)$ be Euler's phi function.  Then $\{\varphi(n)\}$ is not a strong Benford sequence. 
\end{corollary}

Finally, let $f(z) = \sum_{n = 1}^\infty \lambda_f(n) q^n \in S_k^{new}(\Gamma_0(N))$ where $q = e^{2\pi i z}$,  be a newform (i.e., a holomorphic cuspidal normalized Hecke eigenform) of even weight $k$ and trivial nebentypus on $\Gamma_0(N)$ that does not have complex multiplication. We will investigate the sequence of Fourier coefficients $\{ \lambda_f(n)\}$. A classic example of this type of sequence is the Ramanujan tau function $\tau(n)$. Previously, Jameson, Thorner and Ye \cite{JTY} showed that the sequence $\{\lambda_f(p)\}$, where $p$ is prime, does not satisfy Benford's law, but it does follow Benford's law with logarithmic density, which is \footnote{We state only a special case of \cite{JTY}, where the base is 10.}
$$ \lim_{n \rightarrow \infty} \frac{ \sum_{\substack{p \leq n \\ \textrm{the first $K$-digits of $\lambda_f(p)=  S$}}} \frac 1p}{\sum_{p \leq n} \frac{1}{p}} = \log_{10} \left( 1 + \frac 1S\right). $$
Our Theorem \ref{thm:Benfordcriterion} allows us to consider the sequence over natural numbers $n$, not restricted to primes. 

\begin{corollary} \label{cor:newforms} Let $f(z) = \sum_{n = 1}^\infty \lambda_f(n) q^n \in S_k^{new}(\Gamma_0(N))$ where $q = e^{2\pi i z}$,  be a newform 
of even weight $k$ and trivial nebentypus on $\Gamma_0(N)$ that does not have complex multiplication. 
Then $\{\lambda_f(n)\}$ is a strong Benford sequence. 
\end{corollary}

We remark that previously Anderson, Rolen and Stoehr \cite{ARS} proved that the non-zero coefficients of a special family of weakly holomorphic modular forms and certain partition functions are Benford.  We refer the reader there for more precise statements.

Our final result shows that the functions violating Benford's law possess some algebraic structure.

\begin{corollary}\label{cor:group} The collection of multiplicative functions $f\colon \mathbb{N} \to \R - \{0\}$ for which $\{f(n)\}$ is not a strong Benford sequence forms a group under pointwise multiplication.
\end{corollary}

As one illustration of Corollary \ref{cor:group}, we observe that since $\{d(n)\}$ is a strong Benford sequence while $\{\phi(n)\}$ is not, $\{d(n)\phi(n)\}$ is a strong Benford sequence.

\section{The Main Theorem}  \label{sec:MainThm}

Our criterion for real multiplicative functions to satisfy Benford's law is stated in terms of the ``distance'' between two multiplicative functions. We start by defining this distance.

\begin{defi} \label{def:distance}
Let $f$ and $g$ be multiplicative functions taking values in the unit disc $\{z\in \C: |z|\le 1\}$. The distance up to $x$ between $f$ and $g$ is defined to be
$$
\mathbb D (f, g; x) := \left(\sum_{p \leq x} \frac{1 - \mathrm{Re}(f(p) \overline{g(p)} )}{p}\right)^{1/2}.
$$
\end{defi}

This notion of distance comes from the work of Granville and Soundararajan \cite{GS1}. It is known (see \cite[p.\ 364]{GS1}) that the distance function satisfies the triangle inequality, i.e. for any multiplicative function $f, f', g$, and $g'$,
\begin{equation*} 
  \mathbb D(ff', gg'; x) \ \leq \ \mathbb D(f, g; x) + \mathbb D(f', g'; x). 
\end{equation*}

If $\mathbb D(f, g; \infty) < \infty$, then we say that $f$ {\it pretends} to be $g$. 
Let $h\colon \mathbb N \rightarrow \mathbb R - \{0\}$ be a multiplicative function. It is obvious that $f_\ell(n) := e^{2\pi i \ell \log_{10} |h(n)|}$ is also multiplicative.


We will start with the criterion for nonzero multiplicative functions. The criterion for $\{h(n) \}$ to be a strong Benford sequence is that $f_\ell$ does not pretend to be $n^{i\alpha}$ for all $\alpha \in \mathbb R$ and all nonzero integers $\ell$. More precisely, we have the following. 

\begin{thm} \label{thm:Benfordcri2}   Let $h\colon \mathbb N \rightarrow \mathbb R -\{0\}$ be a multiplicative function. Let $f_\ell(n) = e^{2\pi i \ell \log_{10} |h(n)|}$.  $\{h(n)\}$ is a  strong Benford sequence if and only if  $\mathbb D (f_\ell, n^{i\alpha}; \infty) = \infty$ for all $\alpha \in \mathbb R$ and all $\ell \ne 0$.

\end{thm} 

For multiplicative functions which are possibly zero for some $n$, we need an additional condition on the frequency of nonzero terms. 
\begin{thm} \label{thm:Benfordcriterion}  Let $h\colon \mathbb N \rightarrow \mathbb R$ be a multiplicative function. Define 
$$\mathcal T_{N} = |\{ n \leq N \ | \  h(n) \neq 0\}|.$$ 
Suppose $\frac{N}{\mathcal T_N} \ll 1.$  Let $f_\ell(n)$ be defined as in Theorem \ref{thm:Benfordcri2}, and $g_\ell(n) = f_\ell(n)$ if $h(n) \neq 0$ and  is 0 otherwise. Then $\{h(n)\}$ is a strong Benford sequence if and only if  $\mathbb D (g_\ell, n^{i\alpha}; \infty) = \infty$ for all $\alpha \in \mathbb R$ and all $\ell \ne 0$.  
\end{thm}
\begin{rem}
$g_\ell(n)$ is also multiplicative. For $(m, n) = 1$, if both $h(m)$ and $h(n)$ are not zero, we have $$g_\ell(mn) = f_\ell(mn) = f_\ell(m) f_\ell(n) = g_\ell(m) g_\ell(n).$$
Otherwise
$$ g_\ell(m) g_\ell(n) = 0 = g_\ell(mn).$$
\end{rem}

\begin{rem} Theorem \ref{thm:Benfordcri2} is a corollary of Theorem \ref{thm:Benfordcriterion}, where $\mathcal T_N = N$ and $g_\ell(n) = f_\ell(n).$ We state Theorem \ref{thm:Benfordcri2} for clarity. 

\end{rem}

The notation of being strong Benford is closely connected to the notion of being uniformly distributed modulo $1$, which we define formally below.

\begin{defi} A sequence $\{a_k\}$ is uniformly distributed modulo 1 if and only if the fractional parts of all numbers in the sequence distribute uniformly on the interval $[0,1]$, i.e.,
$$ \lim_{n \rightarrow \infty} \frac{|\{ k \leq n: \ a_k \mod 1 \in (a, b) \}|}{n} = b - a,  $$
where $(a, b) \subset [0, 1].$
\end{defi}

Diaconis \cite{Persi} showed that being strong Benford is equivalent to being uniformly distributed mod $1$.  To be precise, we have the following result.
\begin{lem} \label{lem:Benford} A nonzero sequence $\{a_k\}$ is a strong Benford sequence if and only if the sequence $\{\log_{10} |a_k| \}$ is uniformly distributed modulo 1. 
\end{lem}

A classic result of Weyl gives a necessary and sufficient condition for a sequence to be uniformly distributed.

\begin{thm}[Weyl's criterion] \label{thm:Weyl} The sequence $\{ a_k \}$ is uniformly distributed modulo 1 if and only if 
$$ \lim_{N \rightarrow \infty} \frac 1N\sum_{k = 1}^N e^{2\pi i \ell a_k} = 0$$
for all integers $\ell \neq 0.$
\end{thm}

Finally, Hal\'{a}sz's Theorem enables us to understand the averages of multiplicative functions by comparing them to $n^{i\alpha}$. 



\begin{thm}[Hal\'asz's Theorem]\label{thm:Halasz} Let $f$ be a multiplicative function with $|f(n)|\le 1$ for all integers $n$. If $\mathbb{D}(f,n^{i\alpha};\infty) = \infty$ for all $\alpha \in \R$, then $f$ has mean value zero, in the sense that $$ \lim_{x \rightarrow \infty} \frac 1x\sum_{n \leq x} f(n) = 0.$$
Otherwise there is a unique $\alpha \in \R$ with $\mathbb{D}(f,n^{i\alpha};\infty) < \infty$. In that case, $f$ has mean value $0$ if and only if 
\begin{equation}\label{eq:twocondition} f(2^k) = -2^{ik\alpha} \quad\text{for every positive integer $k$}.\end{equation}
\end{thm}

Theorem \ref{thm:Halasz} is essentially contained in Theorem 6.3 on pp.\ 226--227 of Elliott's monograph \cite{Elliott}. The uniqueness of $\alpha$ with $\mathbb{D}(f,n^{i\alpha};\infty) < \infty$ is not explicit in that statement but is proved on p.\ 248 of that reference.

By Weyl's criterion, it suffices to apply Hal\'{a}sz's Theorem to $g_\ell (n)$. As mentioned earlier, Theorem \ref{thm:Benfordcri2} is a simple corollary of \ref{thm:Benfordcriterion}. Hence, we focus on Theorem \ref{thm:Benfordcriterion}.

\subsection{Proof of Theorem \ref{thm:Benfordcriterion}} 

Let $\{h^*(k)\}$ be the subsequence of $\{h(n)\}$ where all zero terms are removed. By Weyl's criterion, $\{h^*(k)\}$ is a strong Benford sequence if and only if 
$$ \lim_{\mathcal T_N \rightarrow \infty} \frac{1}{\mathcal T_N}\sum_{k = 1}^{\mathcal T_N} e^{2\pi i \ell \log_{10}|h^*(k)|} = 0$$
for all nonzero integers $\ell$. However, the function $e^{2\pi i \ell \log_{10}|h^*(k)|}$ is not necessarily multiplicative so we cannot apply Hal\'asz's Theorem. Therefore we add back some zero terms by constructing the function $g_\ell$ defined in Theorem \ref{thm:Benfordcriterion} and write
$$   \frac{1}{\mathcal T_N}\sum_{k = 1}^{\mathcal T_N} e^{2\pi i \ell \log_{10}|h^*(k)|} = \frac{N}{\mathcal T_N}  \frac{1}{N}\sum_{n = 1}^N g_\ell(n) $$
Since $\frac{N}{\mathcal T_N} \ll 1$, $\{h^*(k)\}$ is a strong Benford sequence if and only if 
$$  \lim_{N \rightarrow \infty}\frac{1}{N}\sum_{n = 1}^N g_\ell(n) = 0$$
for all integers $\ell \ne 0$. The ``if'' direction of Theorem \ref{thm:Benfordcriterion} now follows immediately from applying Hal\'{a}sz's Theorem (Theorem \ref{thm:Halasz}) to the functions $g_\ell$. 

For the ``only if'' direction, suppose that $\{h^*(k)\}$ is strong Benford, so that each $g_{\ell}$, with $\ell \ne 0$, has mean value $0$. By Theorem \ref{thm:Halasz}, it is enough to show that there is no nonzero integer $\ell$ and real number $\alpha$ for which $\mathbb{D}(g_{\ell},n^{i\alpha};\infty)  < \infty$ and $g_{\ell}(2^k)=-2^{ik\alpha}$ for all positive integers $k$. If such $\ell$ and $\alpha$ exist, consider $g_{2\ell}=g_{\ell}^2$. By the triangle inequality,
\[ \mathbb{D}(g_{2\ell},n^{2i\alpha};\infty) \le 2\cdot \mathbb{D}(g_{\ell},n^{i\alpha};\infty) < \infty. \]
Also, $g_{2\ell}(2) = (g_{\ell}(2))^2 = 2^{i(2\alpha)} \ne -2^{i(2\alpha)}$. Invoking Hal\'asz's theorem again, we see that $g_{2\ell}$ does not have mean value $0$, a contradiction.

 






\section{Proof of Corollary \ref{cor:poly} -- \ the function $n^a$}
Here, $f_\ell(p) = e^{2\pi i\ell \log_{10} p^a} = p^{\frac{2\pi i \ell a}{\ln 10}}.$ Let $\alpha_\ell = \frac{2\pi  \ell a}{\ln 10}.$ It is easy to see that
$$ \mathbb D(f_\ell, n^{i \alpha_\ell}; \infty)^2 = \sum_{p} \frac{1 - \Rep(f_\ell(p) p^{-i\alpha_\ell})}{p} = 0. $$
By Theorem \ref{thm:Benfordcri2}, $\{ n^a\}$ is not Benford. 

\section{Proof of Corollary \ref{cor:dkn} -- $k$-divisor functions}
For this case,  $f_\ell(n) = e^{2\pi i \ell \log_{10} d_k(n)}$ so $f_\ell(p) = e^{2\pi i \ell \log_{10} k}.$ The corollary will follow if we show the following.
\begin{enumerate}
\item  For $k \neq 10^j$, 
\begin{equation} \label{eqn:distdk} \mathbb D(f_\ell, n^{i\alpha}; \infty)^2 = \sum_{p} \frac{1- \Rep(e^{2\pi i \ell \log_{10} k} p^{-i\alpha})}{p} = \infty
\end{equation}
for all $\alpha \in \mathbb R.$
\item For $k = 10^j,$  $ \mathbb D(f_\ell, 1; \infty) < \infty$ for some $\ell \ne 0$.\end{enumerate}
\vskip 0.2in
First, we consider the case $k = 10^j,$ and we chose $\ell = 1$. 
$$ \mathbb D(f_1, 1; \infty)^2 = \sum_{p} \frac{1- \Rep(e^{2\pi i  \log_{10} 10^j} )}{p} =  \sum_p \frac{1 - 1}{p} = 0. $$
Thus by Theorem \ref{thm:Benfordcriterion}, $\{d_k(n)\}$ is not a Benford sequence.  
\\  
  
Now we consider the case $k \neq 10^j.$  Here, we will use some classical results.  The proof of the Lemma below can be found in \cite[Theorem 2.7]{MV}. 

\begin{lem}[Mertens' Theorem] \label{eqn:sum1p} For $x \geq 2,$ we have
\begin{equation*} 
 \sum_{p \leq x} \frac 1p = \ln \ln x + A + \mathfrak b(x), \ \ \ \ \ \ {where} \ \ \mathfrak b(x) = O\left( \frac{1}{\ln x}\right), 
 \end{equation*}
and 
$$ \sum_{p \leq x} \frac{\ln p}{p} = \ln x + O(1).$$
\end{lem}

We also state some classical bounds on $\zeta(s)$ near the $\tRe s = 1$ line.  We refer the reader to (3.5.1) and (3.11.8) of Titchmarsh's book \cite{Ti}.

\begin{lem} \label{lem:zetaprimeoverzeta} Let $s = \sigma + it$.  There exists some constant $c>0$ such that for  $1 - \frac{c}{\ln (|t| + 4)} < \sigma $ and $|t| > 2$, then 
$$\frac{1}{\ln |t|} \ll \zeta(s) \ll \ln |t|.$$
On the other hand, if  $|t| \leq 2$, 
$$  \zeta(s) \ll \frac{1}{|t|} + O(1).$$
\end{lem}

\subsection{Proof of \eqref{eqn:distdk} when $k \neq 10^j$ }
Note that
$$ \Rep(e^{2\pi i \ell \log_{10} k} p^{-i\alpha}) = \cos(2\pi \ell \log_{10}k) \cos(\alpha \ln p) + \sin(2\pi \ell \log_{10} k )\sin(\alpha \ln p).$$
When $\alpha = 0$, 
$$  \mathbb D(f_\ell, 1; \infty)^2 = \sum_{p} \frac{1- \cos(2\pi \ell \log_{10}k)}{p}. $$

Since $\cos(2\pi \ell \log_{10}k) < 1$ when $k \neq 10^j,$  $\mathbb D(f_\ell, 1; \infty) = \infty$ by Lemma \ref{eqn:sum1p}.

From now on, we focus on $\alpha \neq 0.$ Since $\sum_{p} \frac 1p = \infty$, to prove \eqref{eqn:distdk}, it suffices to show that 
$$  \sum_{p \leq x} \frac{\cos (\alpha \ln p)}{p} = O(1),  \ \ \ \  \textrm{and} \ \ \ \ \  \sum_{p \leq x} \frac{\sin (\alpha \ln p)}{p} = O(1)$$
uniformly in $x$. Since $\Rep(p^{i \alpha}) = \cos(\alpha \ln p) $ and $\textrm{Im}(p^{i\alpha}) = \sin(\alpha \ln p),$  it suffices to prove the following Lemma. 

\begin{lem} \label{lem:sump1}Let $\alpha$ be a fixed nonzero real number. Then 
$$  \sum_{p \leq x} \frac{1}{p^{1 + i\alpha}}=  O_\alpha(1),$$
where the implied constant depends on $\alpha.$ 

\end{lem}

 \begin{proof}

We claim that for any $y \geq 2$,
\begin{equation} \label{eqn:extra1logx}
\sum_{p \leq y} \left(\frac{1}{p^{1 + i\alpha}} - \frac{1}{p^{1 + \frac 1{\ln y} + i\alpha}} \right) = O(1),
\end{equation}
and 
\begin{equation} \label{eqn:sumtoinfty}
\sum_{p \leq y}  \frac{1}{p^{1 + \frac 1{\ln y} + i\alpha}}  =  \ln \zeta\left(1 + \frac 1{\ln y} + i\alpha\right) + O(1).
\end{equation}

\eqref{eqn:extra1logx} follows from  
\est{ \left|\sum_{p \leq y} \left(\frac{1}{p^{1 + i\alpha}} - \frac{1}{p^{1 + \frac 1{\ln y} + i\alpha}} \right) \right| \leq \sum_{p \leq y} \frac{1 - e^{-\frac{\ln p}{\ln y} }}{p} \ll  \frac{1}{\ln y}\sum_{p \leq y} \frac{\ln p}{p} \ll 1}
by Lemma \ref{eqn:sum1p}. For \eqref{eqn:sumtoinfty}, we will use the fact that 
$$ \zeta(s) = \prod_{p} \left( 1 - \frac 1{p^s} \right)^{-1} $$
for $\Rep(s) > 1$, Lemma \ref{eqn:sum1p} and partial summation to derive that 
\est{ \ln \zeta\left(1 + \frac 1{\ln y} + i\alpha\right) - \sum_{p \leq y}  \frac{1}{p^{1 + \frac 1{\ln y} + i\alpha}}  &= 
 \sum_{p > y}  \frac{1}{p^{1 + \frac 1{\ln y} + i\alpha}} + O(1) \\
&\ll  \int_{y}^\infty \frac{1}{u^{\frac 1{\ln y}}} \frac{1}{u\ln u} \> du = O(1).}

If $x \le 2$, the result is trivial.  Otherwise, by \eqref{eqn:extra1logx} and \eqref{eqn:sumtoinfty}, 
\es{ \label{eqn:secondprimesumdk} \sum_{p \leq x} \frac{1}{p^{1 + i\alpha}} &= \ln \zeta\left(1 + \frac 1{\ln x} + i\alpha \right) + O(1)  \\
&\ll_\alpha 1,}
by Lemma \ref{lem:zetaprimeoverzeta}.
 \end{proof}

\section{Proof of Corollary \ref{cor:phi} -- Euler's phi function}
For Euler's phi function, $f_{\ell}(p) = e^{2\pi \ell i \log_{10}(p-1)} = (p - 1)^{\frac{2\pi i \ell}{\ln 10}  }$. Let $\alpha = \frac{2\pi}{\ln 10}$. The Corollary follows from
$$\mathbb D(f_1, n^{i\alpha}; \infty )^2 < \infty. $$
Now,
\est{
\mathbb D(f_1, n^{ i\alpha}; \infty )^2 =  \sum_{p} \frac{1 - \Rep(( p -1)^{i\alpha}p^{-i\alpha})}{p} = \sum_{p} \frac{1 - \cos\left(\alpha\ln\left( 1 - \frac 1p\right)\right)}{p} . }

For large prime $p$ such that $|\frac{\alpha}p| < \frac 12,$
$$ \cos\left(\alpha\ln\left( 1 - \frac 1p\right)\right) = \cos \left( \alpha \left( - \frac 1p + O\left( \frac 1{p^2} \right)\right)\right) = 1 + O\left(\frac{\alpha^2}{p^2}\right).$$

Let $J = \max\{10, 2|\alpha| \}$. Thus
\est{
 \sum_{p} \frac{1 - \cos\left(\alpha\ln\left( 1 - \frac 1p\right)\right)}{p} &\ll \sum_{p \leq J} \frac{1}{p} + \sum_{p > J} \frac{\alpha^2}{p^3} \ll \ln \ln J + \frac{\alpha^2}{J^2} = O_\alpha(1). }
Thus, we conclude that $\{\varphi(n)\}$ is not a strong Benford sequence.  

\section{Proof of Corollary \ref{cor:newforms} -- Hecke eigenvalues of newforms}

Contrary to the previously considered multiplicative functions, $\lambda_f(n)$ might be $0$ for some $n.$  Our Corollary follows from Theorem \ref{thm:Benfordcriterion} and the following lemmas.
\begin{lem} \label{lem:NoverTN} Let $\mathcal T_N$ be defined as in Theorem \ref{thm:Benfordcriterion}. Then 
$$ \frac{N}{\mathcal T_N} \ll 1. $$

\end{lem}
\begin{lem} \label{eqn:distnewform} Let $g_\ell(n) = e^{2\pi \i \ell \log_{10}|\lambda_f(n)|}$ if $\lambda_f(n) \neq 0$ and 0 otherwise. Then  
$$ \mathbb D(g_\ell, n^{i\alpha}; \infty) = \infty$$
for all $\alpha \in \mathbb R$ and $\ell \neq 0.$
\end{lem}



Lemma \ref{lem:NoverTN} follows immediately from the work of Serre stated in Theorem \ref{thm:serre} below.  We start by stating some now well known properties of $\lambda_f(p)$. 

\subsection{Properties of $\lambda_f(p)$} Weil conjectured that for all primes $p$, 
$$ |\lambda_f(p)| \leq 2p^{\frac{k-1}{2}},$$
and this is proven by Deligne \cite{Deligne}. Therefore for each $p$, there exists unique $\theta_p \in [0, \pi]$ such that 
$$ \lambda_f(p) = 2 p^{\frac{k-1}{2}} \cos \theta_p.$$

Sato and Tate studied the distribution of ${\cos \theta_p}$, varying through $p$ for newforms $f$ associated with elliptic curves. They conjectured that $\{\cos \theta_p\}$ is equidistributed in $[-1, 1]$ with respect to a certain measure. Later Barnet-Lamb, Geraghty, Harris, and Taylor \cite{BGHT} proved the conjecture, and in fact they generalized it for the larger class of Hecke newforms, which we state below.  

\begin{thm}[The Sato-Tate Conjecture] \label{thm:satotate} Let $f(z) = \sum_{n = 1}^\infty \lambda_f(n)q^n \in S_k^{new}(\Gamma_0(N))$ be a
newform of even weight $k \geq 2$ without complex multiplication. Let $F\colon [-1,1] \rightarrow \mathbb C$ be a Riemann-integrable function. Then
$$ \lim_{x \rightarrow \infty} \frac{1}{\pi(x)} \sum_{p \leq x} F(\cos(\theta_p)) = \int_{-1}^{1} F(t) \, \mathrm{d}\mu_{ST}, $$
where $\pi(x) = \# \{ p \leq x  \}$ is the prime counting function, and 
$$ \mathrm{d}\mu_{ST} = \frac{2}{\pi} \sqrt{1 - t^2} \, \mathrm{d}t. $$


\end{thm}
The next result was proven by Serre in \cite{Serre} (see Theorem 15, the accompanying Corollary 2 and Theorem 16).  Interested readers may also peruse the later works of Wan \cite{Wan} for an improvement, and of Murty \cite{Murty} and Thorner and Zaman \cite{TZ} for further refinements.

\begin{thm} \label{thm:serre}Let $f(z)$ be defined as in Theorem \ref{thm:satotate}. Then for any $\delta < 1/2$,
$$\#\{p\le x: \lambda_f(p) = 0\} \ll_{f, \delta} \frac{x}{\log^{1+\delta} x}.
$$Hence
$$\sum_{\substack{p\\ \lambda_f(p) = 0}} \frac 1p < \infty,
$$and
$$\#\{n\le x: \lambda_f(n) \neq 0\} \gg x.
$$
\end{thm}

Note that Serre's result immediately implies Lemma \ref{lem:NoverTN}.  We prove Lemma \ref{eqn:distnewform} below.

\subsection{Proof of Lemma \ref{eqn:distnewform} }
From the definition of $g_\ell(n),$
$$ \mathbb D(g_\ell, n^{i\alpha}; \infty)^2 = \sum_{\substack{p \\ \lambda_f(p) \neq 0}} \frac{1 - \Rep(e^{2\pi i \ell \log_{10}|\lambda_f(p)|} p^{-i\alpha})}{p} + \sum_{\substack{p \\ \lambda_f(p) = 0}} \frac{1}{p}. $$
We will deduce Lemma \ref{eqn:distnewform} from the following Lemma.

\begin{lem} \label{lem:Relambdap} Let $\beta = \frac{2\pi \ell}{\ln 10}$ for $\ell \neq 0$, and $\gamma = \frac{k-1}{2} \beta - \alpha.$   Then 
\est{\left|\sum_{\substack{p \leq x \\ \lambda_f(p) \neq 0}} \frac{ \tRe \left({|2\cos\theta_p|^{i\beta } p^{i \gamma }}\right)}{p} \right| = (\mathcal K + o(1)) \ln\ln x }
where 
\begin{align*}
\mathcal K = 
\begin{cases}
\frac{4}{\pi} \int_0^1 \cos(\beta \ln (2y)) \sqrt{1 - y^2}\, \mathrm{d}y &\textup{ if $\gamma = 0$ } \\
0 &\textup{ otherwise.}
\end{cases}
\end{align*}
\end{lem} 

There is an interval of positive measure inside $[0, 1]$ on which $\cos(\beta \ln (2y)) < 1$, so by continuity,

\es{\label{eqn:setupcalK}&\frac{4}{\pi} \int_{0}^1 \cos(\beta \ln (2y)) \sqrt{1 - y^2} \, \mathrm{d}y \\
&<  \frac{4}{\pi} \int_{0}^1 \sqrt{1 - y^2} \, \mathrm{d}y = 1,} so $\mathcal K <1$.  Thus, by Lemma \ref{lem:Relambdap}, 

\est{\mathbb D(g_\ell, n^{i\alpha}; x)^2 &\ge \sum_{\substack{p \leq x \\ \lambda_f(p) \ne 0}} \frac{1 - \Rep \left({|2\cos\theta_p|^{i\beta } p^{i \gamma }}\right)}{p} \gg \ln \ln x }
since $\sum_{p \leq x} \frac 1p \sim \ln \ln x$ (Lemma \eqref{eqn:sum1p}), and so
$$ \mathbb D(g_\ell, n^{i\alpha}; \infty) = \infty,$$
as required for Lemma \ref{eqn:distnewform}.  We now prove Lemma \ref{lem:Relambdap}.

\begin{proof}
Note that 
$$\Rep \left({|2\cos\theta_p|^{i\beta } p^{i \gamma }}\right) = \cos(\beta \ln (2|\cos \theta_p|)) \cos(\gamma \ln p) - \sin(\beta \ln (2 |\cos \theta_p|)) \sin(\gamma \ln p). $$

Thus, it suffices to show that
\es{ \label{eqn:cosbigcos}\sum_{\substack{p \leq x \\ \lambda_f(p) \neq 0 } } \frac{\cos(\beta \ln (2|\cos \theta_p|)) \cos(\gamma \ln p)}{p} = (\mathcal K + o(1)) \ln\ln x }
 and 
\es{\label{eqn:cosbigsin} \sum_{\substack{p \leq x \\ \lambda_f(p) \neq 0 } } \frac{\sin(\beta \ln (2 |\cos \theta_p|)) \sin(\gamma \ln p)}{p} = o(\ln\ln x) . }
We proceed to prove \eqref{eqn:cosbigcos}, noting that \eqref{eqn:cosbigsin} follows similarly.

By partial summation, the left-hand side of \eqref{eqn:cosbigcos} is
\es{\label{eqn:cosafterPS}&\left(\sum_{\substack{ p \leq x \\\lambda_f(p) \neq 0 } }  \cos(\beta \ln (2|\cos \theta_p|))  \right) \frac{\cos(\gamma \ln x)}{x} - \int_2^x \left( \sum_{\substack{p \leq t \\ |\cos \theta_p| > 0 } }  \cos(\beta \ln (2|\cos \theta_p|)) \right)\, \mathrm{d} \frac{\cos(\gamma \ln t)}{t} }

Let $I = [-1, 0) \cup (0, 1]$ and $\chi_I(x)$ denote the characteristic function of $I$, which is $1$ if $x \in I$, and $0$ otherwise. Note that $\cos(\beta \ln(2|y|))$ is Riemann integrable over $[-1, 1]$, the singularity at $y=0$ being benign.

Thus, by the Sato-Tate conjecture (Theorem \ref{thm:satotate}), 
\est{\lim_{x \rightarrow \infty} \frac{1}{\pi (x)} \sum_{\substack{ p \leq x \\ \lambda_f(p) \neq 0 } }  \cos(\beta \ln (2|\cos \theta_p|))  &=  \lim_{x \rightarrow \infty} \frac{1}{\pi (x)} \sum_{\substack{ p \leq x } }  \cos(\beta \ln (2|\cos \theta_p|)) \chi_{I} (\cos \theta_p) \\
&= \frac{4}{\pi} \int_{0}^1 \cos(\beta \ln (2y)) \sqrt{1 - y^2} \> \mathrm{d}y =: \mathcal K_0,}
since the integrand is even.  By the prime number theorem, $\pi(x) \sim \frac{x}{\ln x}$, we obtain that \eqref{eqn:cosafterPS} is 
\es{\label{eqn:cosafterPS2} &\left(\mathcal K_0 \pi(x) + o(\pi(x))  \right) \frac{\cos(\gamma \ln x)}{x} - \int_2^x \left( \mathcal K_0 \pi (t) + o(\pi(t)) \right) \, \mathrm{d} \frac{\cos(\gamma \ln t)}{t} \\
&= \mathcal K_0 \left( \pi (x) \frac{\cos(\gamma \ln x)}{x} - \int_2^x   \pi(t)\, \mathrm{d}\frac{\cos(\gamma \ln t)}{t} \right) + o\left(\frac{1}{\ln x} + \int_2^x \frac{1}{t\ln t} \, \mathrm{d}t\right) \\
&= \mathcal K_0 \sum_{p \leq x} \frac{\cos (\gamma \ln p)}{p} + o(\ln\ln x).  }

We consider two cases.
\subsubsection*{Case 1: $\gamma \neq 0$}
Here, $\mathcal K = 0$, and Lemma \ref{lem:sump1} implies that
$$ \sum_{p \leq x} \frac{\cos (\gamma \ln p)}{p} = O_\gamma(1).$$
Thus \eqref{eqn:cosafterPS2} is $o(\ln \ln x)$, so we derive the desired result when $\gamma \neq 0$.

\subsubsection*{Case 2: $\gamma = 0$} 
Here, $\mathcal K = \mathcal K_0$, and we have 

\es{ \label{eqn:cosbigcoscase2}\sum_{\substack{p \leq x \\ \lambda_f(p) \neq 0 } } \frac{\cos(\beta \ln (2|\cos \theta_p|)) \cos(\gamma \ln p)}{p} = \mathcal K \sum_{p \leq x}\frac{1}{p} + o(\ln\ln x) = \mathcal K \ln \ln x + o(\ln \ln x),}
as desired.

\end{proof}

\section{Proof of Corollary \ref{cor:group} -- the group of non-Benford multiplicative functions}

Let $h: \mathbb{N}\to\R - \{0\}$ be a multiplicative function. Since  $\mathbb{D}(e^{2\pi i \ell \log_{10} |h(n)|}, n^{i\alpha};\infty) = \mathbb{D}(e^{2\pi i \ell \log_{10} |1/h(n)|}, n^{-i\alpha};\infty)$, it is immediate from Theorem \ref{thm:Benfordcri2} that $h$ is Benford if and only if $1/h$ is Benford.

Thus, it is enough to prove that if $h, h'\colon \mathbb{N}\to\R - \{0\}$ are multiplicative functions for which neither $\{h(n)\}$ nor $\{h'(n)\}$ is a strong Benford sequence, then $\{h(n)h'(n)\}$ is also not a strong Benford sequence. Let 
$$ f_{\ell}(n) = e^{2\pi i \ell \log_{10}|h(n)|}, \quad f'_{\ell}(n) = e^{2\pi i \ell \log_{10}|h'(n)|}, \quad\text{and}\quad F_{\ell}(n) = e^{2\pi i \ell \log_{10}|h(n) h'(n)|}.$$

By Theorem \ref{thm:Benfordcri2}, there are nonzero integers $\ell, \ell'$ as well as real numbers $\alpha, \alpha'$, for which $\mathbb{D}(f_{\ell}, n^{i\alpha}; \infty) < \infty$ and $\mathbb{D}(f'_{\ell'}, n^{i\alpha'}; \infty) < \infty$. By the triangle inequality,
\begin{align*} \mathbb{D}(F_{\ell\ell'}, n^{i(\alpha\ell' + \alpha'\ell)}; \infty) &= \mathbb{D}(f_{\ell}(n)^{\ell'} f'_{\ell'}(n)^{\ell}, n^{i(\alpha\ell' + \alpha'\ell)}; \infty) \\ &\le |\ell'|\cdot \mathbb{D}(f_{\ell}(n),n^{i\alpha};\infty) + |\ell| \cdot \mathbb{D}(f'_{\ell'}(n),n^{i\alpha'};\infty) < \infty. \end{align*}
Applying Theorem \ref{thm:Benfordcri2} once more,  the sequence $\{h(n)h'(n)\}$ is not strongly Benford.

\section*{Acknowledgement}
V.C. and X.L. acknowledge support from a Simons Foundation Collaboration Grant for Mathematicians. V.C. is also supported by NSF grant DMS-2101806, and P.P. is supported by NSF grant DMS-2001581.



\end{document}